\documentclass[12pt,reqno]{amsart}
\usepackage{amsaddr}

\usepackage{hyperref}

\usepackage{a4wide}
\usepackage{amsthm}
\usepackage{amsmath}
\usepackage{amssymb}
\usepackage[all,cmtip]{xy}
\usepackage{color}
\usepackage[T1]{fontenc}
\usepackage{enumerate}
\usepackage{verbatim}
\usepackage{marginnote}

\theoremstyle{plain}
\newtheorem{thm}{Theorem}
\newtheorem{lem}[thm]{Lemma}
\newtheorem{prop}[thm]{Proposition}
\newtheorem{cor}[thm]{Corollary}
\theoremstyle{definition}
\newtheorem{defn}[thm]{Definition}
\newtheorem{exmp}[thm]{Example}

\newtheorem{rem}[thm]{Remark}

\DeclareMathOperator{\id}{id}

\newcommand{\N}{\mathbb{N}}

\begin{document}

\title[Simplicity of Ore monoid rings]{Simplicity of Ore monoid rings}

\author{Patrik Nystedt}
\address{Department of Engineering Science,
University West,
SE-46186 Trollh\"{a}ttan, Sweden}

\author{Johan \"{O}inert}
\thanks{Corresponding author: Johan \"{O}inert}
\address{Department of Mathematics and Natural Sciences,
Blekinge Institute of Technology,
SE-37179 Karlskrona, Sweden}

\author{Johan Richter}
\address{Academy of Education, Culture and Communication,
M\"{a}lardalen University,
Box 883, SE-72123 V\"{a}ster\aa s, Sweden}

\email{patrik.nystedt@hv.se; johan.oinert@bth.se; johan.richter@mdh.se}

\subjclass[2010]{17D99, 17A36, 17A99}
\keywords{non-associative Ore extension, iterated Ore extension, generalized monoid ring, Ore monoid ring, differential monoid ring, simple ring, outer derivation.}

\begin{abstract}
Given a non-associative unital ring $R$, a 
monoid $G$ and a set $\pi$ of additive maps $R \rightarrow R$,
we introduce the Ore monoid ring $R[\pi ; G]$,
and, in a special case, the differential monoid ring.
We show that these structures generalize, in a natural way,
not only the classical Ore extensions and differential polynomial 
rings, but also the constructions, introduced by Cojuhari,
defined by so-called $D$-structures $\pi$.
Moreover, for commutative monoids, we give necessary and sufficient conditions
for differential monoid rings to be simple.
We use this in a special case to 
obtain new and shorter proofs of classical 
simplicity results for differential polynomial rings in several variables
previously obtained by Voskoglou and Malm by other means.
We also give examples of new Ore-like structures
defined by finite commutative monoids. 
\end{abstract}

\maketitle

\pagestyle{headings}

%\tableofcontents

\section{Introduction}\label{sec:one}

\subsection{Background}

Suppose that $R$ is an associative unital ring
with multiplicative identity 1.
In \cite{ore1933} Ore introduced a version
of skew polynomial rings $S = R[x ; \sigma , \delta]$,
nowadays called \emph{Ore extensions}, that have
become a very important construction in ring theory.
Recall that $S$ is
defined to be the polynomial ring $R[x]$ as a left $R$-module,
equipped with a new associative multiplication induced by the relations
$x r = \sigma(r) x + \delta(r)$, for $r \in R$,
where $\sigma : R \to R$ is a ring endomorphism respecting 1
and $\delta : R \to R$ is a $\sigma$-derivation on $R$, 
that is $\delta$ is an additive map
satisfying $\delta(rs)=\sigma(r)\delta(s)+\delta(r)s$ for all $r,s \in R$.
In the special case when $\sigma = {\rm id}_R$, then
$S$ is called a \emph{differential polynomial ring} and $\delta$
is called a \emph{derivation}.
Ore extensions
play an important role when
investigating cyclic algebras, enveloping rings of solvable Lie algebras,
and various types of graded rings 
such as group rings and crossed products 
(see e.g. \cite{cohn1977,jacobson1999,mcconell1988,rowen1988}).
They are also a natural source of
examples and counter-examples in ring theory 
(see e.g. \cite{bergman1964,cohn1961}).
Furthermore, special cases
of Ore extensions are used as
tools in different analytical settings,
such as differential-, pseudo-differential and
fractional differential operator rings
\cite{goodearl1983} and 
$q$-Heisenberg algebras \cite{hellstromsilvestrov2000}.

Many different properties of Ore extensions, such as
when they are integral domains, principal domains,
prime or noetherian have been studied by numerous authors
(see e.g. \cite{cozzens1975,mcconell1988} for surveys, or the articles
\cite{cauchon1979,lamleroy2000,LLM1997,leroymatczuk1991,leroymatczuk1992}).
For a moment, let us focus on the
simplicity of 
a differential polynomial ring $S = R[x ; \id_R , \delta]$.
Recall that $\delta$ is called {\it inner}
if there is $a \in R$ such that
$\delta(r) = ar - ra$, for $r \in R$.
In that case we write $\delta = \delta_a$.
If $\delta$ is not inner, then $\delta$ is called {\it outer}. 
We let the characteristic of a
ring $R$ be denoted by ${\rm char}(R)$. 
In an early article by Jacobson \cite{jacobson1937}
it was shown that if $\delta$ is outer and $R$ is a division ring
with ${\rm char}(R)=0$, then $S$ is simple.
The case of positive characteristic is more complicated 
and $S$ may contain non-trivial ideals.
In fact, Amitsur \cite{amitsur1950,amitsur1957} has shown that 
if $R$ is a simple ring with ${\rm char}(R)= p > 0$, 
then every non-zero ideal of $D$ is generated by a polynomial, all of whose
monomials have degrees which are multiples of $p$.
It should be noticed that simplicity of $R$ is not a necessary condition
for simplicity of $S = R[x ; \id_R , \delta]$.
Consider e.g. the well known example of the first Weyl algebra
where $R = K[y]$, $K$ is a field with ${\rm char}(K)=0$ and 
$\delta$ is the usual derivative on $R$
(for more details, see e.g. \cite[Example 1.6.32]{rowen1988}).
However, $\delta$-simplicity of $R$ is always a necessary
condition for simplicity of $S$ 
(see \cite[Lemma 4.1.3(i)]{jordan1975}).
Recall that an ideal $I$ of $R$ is called \emph{$\delta$-invariant}
if $\delta(I) \subseteq I$.
The ring $R$ is called \emph{$\delta$-simple} if $\{ 0 \}$ and $R$
are the only $\delta$-invariant ideals of $R$.
Jordan \cite{jordan1975} (and Cozzens and Faith \cite{cozzens1975}
in a special case) has generalized the results of Amitsur \cite{amitsur1957}
to the case of $\delta$-simple $R$.
In \cite{oinert2013} \"{O}inert, Richter and Silvestrov
have shown that $S = R[x ; \id_R , \delta]$ is simple 
if and only if $R$ is $\delta$-simple and $Z(S)$ is a field.
Simplicity results for differential polynomial rings in finitely many variables
have been obtained by Voskoglou \cite{voskoglou1985} and Malm \cite{malm1988}.
In \cite{nystedtoinertrichter} Nystedt, \"{O}inert and Richter introduced a variant
of Ore extensions $R[x ; \sigma , \delta]$ for non-associative 
(i.e. not necessarily associative) rings $R$
and weaker types of derivations. In loc. cit. simplicity results were obtained
for these more general types of differential polynomial rings $R[x ; {\rm id}_R , \delta ]$,
thereby generalizing the results of \cite{amitsur1957}, \cite{jordan1975} and \cite{oinert2013}.

Many different generalizations of Ore extensions have been introduced.
In \cite{smits1968} Smits considered skew polynomial rings $R[x]$
with a commutation rule defined by 
$x r = r_1 x + r_2 x^2 + \ldots + r_n x^n$,
where $n \geq 2$ is a fixed integer and $r \in R$.
In loc. cit. it was shown that the maps 
$\delta_i : r \mapsto r_i$, $i \in \{1,\ldots,n\}$, are additive.
In the special case when $\delta_i = 0$, for $i \in \{2,\ldots,n\}$,
one retrieves the skew polynomial ring $R[x ; \sigma , 0]$
where $\sigma = \delta_1$.
Generalizations in several variables have also been considered.
Let $\mathbb{N}$ denote the set of non-negative integers.
If $n \in \mathbb{N}$ and $\delta_0,\ldots,\delta_n$ is a set of 
commuting derivations on $R$, then one can define the 
differential operator ring $R[x_0,\ldots,x_n ; \delta_0 , \ldots , \delta_n]$
as the ordinary polynomial ring 
$R[x_0,\ldots,x_n]$ subject to the relations 
$x_i x_j = x_j x_i$ and $x_i r = r x_i + \delta_i(r)$,
for $i,j \in \{ 0,\ldots,n \}$ and $r \in R$ (see e.g. \cite{voskoglou1985}). 
In \cite{hasse1937} Hasse and Schmidt defined a higher derivation of rank $m \in \mathbb{N}$
as a sequence $d_0,d_1,\ldots,d_m$ of additive endomorphisms of $R$
satisfying the conditions $d_n(ab) = \sum_{i+j=n} d_i(a) d_j(b)$,
for $n \in \{ 0 , \ldots , m \}$, and $a,b \in R$.
If the $d_i$'s commute, then one can define a differential operator ring
$R[x_0,\ldots,x_m ; d_0 , \ldots d_m]$ as the ordinary polynomial ring 
$R[x_0,\ldots,x_m]$ subject to the relations 
$x_i x_j = x_j x_i$ and $x_i r = \sum_{k=0}^i d_k(r) x_{i-k}$,
for $i,j \in \{ 0,\ldots,m \}$ and $r \in R$ (see e.g. \cite{malm1988}).
In \cite{cojuhari2007,cojuhari2012,cojuhari2018} Cojuhari and Gardner
introduced skew versions $R[G;\pi]$ of monoid rings in the following sense.
Let $G$ denote a monoid with identity element $e$.
For all $a,b \in G$, let $\pi^a_b : R \rightarrow R$ be a function.
Following \cite[Definition 1.1]{cojuhari2012} we say that 
$\pi = \{ \pi^a_b \}_{a,b \in G}$ is a \emph{$D$-structure on $R$} if the following axioms hold:
\begin{itemize}

\item[(D0)] For all $a \in G$ and all $r \in R$, we have $\pi^a_b(r)=0$ 
for  all but finitely many $b \in G$.

\item[(D1)] $\pi^e_e = {\rm id}_R$ and for all
$a \in G \setminus \{ e \}$, we have that $\pi^e_a = 0$.

\item[(D2)] For all $a,b \in G$, $\pi^a_b(1)$ equals
the Kronecker delta function $\delta_{a,b}$.

\item[(D3)] For all $a,b \in G$, the map $\pi^a_b$ is additive.

\item[(D4)] For all $a,b,c \in G$, we have
$\pi^{ab}_c = \sum_{de = c} \pi^a_d \circ \pi^b_e$.

\item[(D5)] For all $a,b \in G$ and all $r,s \in R$, we have
$\pi^a_b(rs) = \sum_{c \in G} \pi^a_c(r) \pi^c_b(s)$.

\end{itemize}
The \emph{generalized monoid ring} $R[G;\pi]$ is defined to be the
set of formal sums $\sum_{a \in G} r_a x^a$,
where $r_a \in R$ is zero for all but finitely many $a \in G$.
The addition on $R[G;\pi]$ is defined pointwise and the 
multiplication on $R[G;\pi]$ is defined by the biadditive 
extension of the relations
\begin{equation}\label{eq:multirule}
(r x^a) (s x^b) = \sum_{c \in G} r \pi^a_c(s) x^{cb}
\end{equation}
for $r,s \in R$ and $a,b \in G$.
It is not difficult to verify that
$R[G;\pi]$ is an associative unital ring (see \cite[Section 4]{cojuhari2007}).
Notice that we recover the classical Ore extension $R[x ; \sigma , \delta]$ 
if we put $G=\N$ and let $\pi^a_b$ be the sum of all the
${a \choose b}$ possible compositions of $b$
copies of $\sigma$ and $a-b$ copies of $\delta$ in arbitrary order
(see \cite[Equation (11)]{ore1933} and \cite[Example 2]{cojuhari2007}).
We also retrieve the differential operator rings defined by a higher derivation
$d_0,d_1,d_2,\ldots,d_m$ which is iterative, that is such that
$d_a \circ d_b = {a+b \choose a} d_{a+b}$ for all $a$ and $b$,
if we put $\pi^a_b = (a! / b!) d_{a-b}$, if $a \geq b$, and
$\pi^a_b = 0$, otherwise (see \cite[Section 1]{cojuhari2012}).
Furthermore, crossed products, group rings and the construction
of Smits \cite{smits1968} mentioned above can be obtained from this
construction (see \cite[Section 6]{cojuhari2007}).

\subsection{Outline}

In this article, we consider Ore monoid rings $R[G;\pi]$ in a context which is 
more general than the one studied by Cojuhari \cite{cojuhari2007}. Indeed, we allow $R$ to be non-associative and do not necessarily require (D5) to hold.
Consider the following axioms:
\begin{itemize}

\item[(D6)] For all $a \in G$, the equality $\pi^a_a = {\rm id}_R$ holds.

\item[(D7)] For all $a,b \in G$, the map $\pi^a_b$ is left $R^G$-linear.

\item[(D8)] For all $a,b \in G$, the map $\pi^a_b$ is right $R^G$-linear.

\end{itemize}
Here $R^G$ is defined to be the set of $r \in R$ such that for all $a,b \in G$,
$\pi^a_b(r) = \delta_{a,b} r$.
We say that $\pi$ is a \emph{(unital) $G$-derivation on $R$} if (D0)$-$(D4) (and (D6)) hold.
In that case,
$R[G;\pi]$ is called an \emph{Ore monoid ring}
(\emph{differential monoid ring}).
The addition on $R[G;\pi]$ is defined in the same way as for the generalized monoid ring,
and the multiplication is defined by \eqref{eq:multirule}.
We say that $\pi$ is a \emph{strong (unital) $G$-derivation on $R$}
if also (D7) or (D8) holds.
Moreover, we say that $\pi$ is a \emph{(unital) $D$-structure on $R$} if (D0)$-$(D5) (and (D6)) hold. 
In that case,
$R[G;\pi]$ is called a \emph{classical
Ore monoid ring} (\emph{classical differential monoid ring}).
It follows from the examples in \cite{nystedtoinertrichter} (for the case $G = \mathbb{N}$) 
that the following inclusions hold:
$$\{ \mbox{classical Ore monoid rings} \} \subsetneq 
\{ \mbox{Ore monoid rings} \}$$
$$\{ \mbox{classical differential monoid rings} \} \subsetneq 
\{ \mbox{differential monoid rings} \}$$
The main objective of this article is to extend the 
simplicity results in \cite{nystedtoinertrichter}
to Ore monoid rings. The secondary objective is
to apply our extended results to particular cases of monoids to obtain 
new proofs of classical simplicity results for iterated Ore extensions
as well as showing simplicity results for new Ore-like structures.

In Section \ref{Sec:Prel},
we show some results concerning commutativity and
associativity in Ore monoid rings 
(see Propositions~\ref{fixedring}--\ref{parenthesis}).
At the end of the section,
we describe certain elements
in the center of Ore monoid rings,
in the case when $G$ is commutative
(see Corollary~\ref{corcenter}).

In Section~\ref{Sec:Simplicity}, we show the main result
of this article (see Theorem~\ref{newmaintheorem}) which gives
us necessary and sufficient conditions for simplicity
of strong differential monoid rings 
$R[G;\pi]$ under the hypothesis that $G$ is commutative and that there is a well-order $\preceq$
on $G$ with the property that for all $a,b \in G$ with 
$b \npreceq a$, the equality $\pi_b^a = 0$ holds.

In Section~\ref{Sec:Malm+Voskoglou}, we apply the main result
of the previous section to the monoid $\mathbb{N}^{(I)}$
of functions $f$ from a well-ordered set $I$ to $\mathbb{N}$,
satisfying $f(i)=0$ for all but finitely many $i \in I$
(see Theorem~\ref{cornewtheorem}).
Using this result, we obtain generalizations of classical simplicity 
results by Voskoglou \cite{voskoglou1985} and Malm \cite{malm1988}
for differential polynomial rings in finitely many variables
(see Theorem~\ref{maintheorem0} and Theorem~\ref{maintheoremp}).

\section{The center}\label{Sec:Prel}

Throughout this article, unless otherwise stated, 
$R$ denotes a unital non-associative ring and $S = R[G ; \pi]$
denotes an Ore monoid ring.
Recall that the \emph{commutator} $[\cdot,\cdot] : R \times R \rightarrow R$ 
and the \emph{associator} $(\cdot,\cdot,\cdot) : R \times R \times R \rightarrow R$ 
are defined by $[r,s]=rs-sr$ and
$(r,s,t) = (rs)t - r(st)$ for all $r,s,t \in R$, respectively.
The \emph{commuter} of $R$, denoted by $C(R)$,
is the subset of $R$ consisting
of all elements $r \in R$ such that $[r,s]=0$,
for all $s \in R$.
The \emph{left}, \emph{middle} and \emph{right nucleus} of $R$,
denoted by $N_l(R)$, $N_m(R)$ and $N_r(R)$, respectively, are defined by 
$N_l(R) = \{ r \in R \mid (r,s,t) = 0, \ \mbox{for} \ s,t \in R\}$,
$N_m(R) = \{ s \in R \mid (r,s,t) = 0, \ \mbox{for} \ r,t \in R\}$, and
$N_r(R) = \{ t \in R \mid (r,s,t) = 0, \ \mbox{for} \ r,s \in R\}$.
The \emph{nucleus} of $R$, denoted by $N(R)$,
is defined to be equal to $N_l(R) \cap N_m(R) \cap N_r(R)$.
From the so-called \emph{associator identity} 
$u(r,s,t) + (u,r,s)t + (u,rs,t) = (ur,s,t) + (u,r,st)$,
which holds for all $u,r,s,t \in R$, it follows that
all of the subsets $N_l(R)$, $N_m(R)$, $N_r(R)$ and $N(R)$
are associative subrings of $R$.
The \emph{center} of $R$, denoted by $Z(R)$, is defined to be equal to the 
intersection $N(R) \cap C(R)$.
It follows immediately that $Z(R)$ is an associative, unital
and commutative subring of $R$.

\begin{rem}\label{rem:extendPI} 
If $a,b\in G$, then the additive map $\pi_b^a : R \to R$ 
may be extended to an additive map $\tilde{\pi}_b^a : S \to S$
by letting $\tilde{\pi}_b^a$ be the additive extension of the rule
$\tilde{\pi}_b^a ( r x^c ) = \pi_b^a(r) x^c$,
for $c \in G$ and $r \in R$.
It is worth pointing out that, apart from additivity, $\tilde{\pi}_b^a$
need not enjoy any of the properties (e.g. (D4) or (D5)) of the original map $\pi_b^a$.

By defining $S^G$ analogously to how $R^G$ was defined,
it is clear that $S^G = \sum_{a \in G} R^G x^a$. 
\end{rem}

\begin{prop}\label{fixedring}
If $\pi$ is strong, then the following assertions hold:
\begin{itemize} 

\item[(a)] $R^G$ is a subring of $R$;

\item[(b)] $Z(R)^G$ is a commutative subring of $R$.

\end{itemize}
\end{prop}

\begin{proof}
By assuming (D7) or (D8), the proof becomes straightforward.
\end{proof}

\begin{prop}\label{xanmnr}
If $a \in G$, then $x^a \in N_r(S) \cap N_m(S)$.
\end{prop}

\begin{proof}
Take $r,t \in R$ and $b,c \in G$. First we show that $x^a \in N_r(S)$:
$(rx^b , tx^c , x^a) = 
(rx^b \cdot tx^c)x^a - rx^b (tx^c \cdot x^a) =
\left( \sum_{d \in G} r \pi_d^b(t) x^{dc} \cdot x^a \right) - rx^b \cdot tx^{ca}  
= \sum_{d \in G} r \pi_d^b(t) x^{dca} -  
\sum_{d \in G} r \pi_d^b(t) x^{dca} = 0.$
Next we show that $x^a \in N_m(S)$:
$(rx^b , x^a , tx^c) = 
(rx^b \cdot x^a) t x^c - rx^b (x^a \cdot t x^c) = 
r x^{ba} \cdot tx^c - \sum_{d \in G} rx^b \cdot \pi_d^a(t) x^{dc} 
= \sum_{g \in G} r \pi_g^{ba}(t) x^{gc} -
\sum_{d,f \in G} r (\pi_f^b \circ \pi_d^a)(t) x^{fdc} 
= [(D4)] = \sum_{g \in G} \sum_{fd = g} r (\pi_f^b \circ \pi_d^a)(t) x^{gc} -
\sum_{d \in G} \sum_{f \in G} r (\pi_f^b \circ \pi_d^a)(t) x^{fdc} = 0.$
\end{proof}

Throughout the rest of this section, we shall assume that $G$ is commutative.

\begin{prop}\label{commute}
If $s \in S^G$ and $a \in G$, then $[s,x^a]=0$.
\end{prop}

\begin{proof}
Suppose that $s = \sum_{b \in G} r_b x^b \in S^G$.
Then, for each $b\in G$, $r_b \in R^G$.
Thus
$[s , x^a] = 
s x^a - x^a s = 
\sum_{b \in G} r_b x^{ba} - 
\sum_{b \in G} \sum_{c \in G} \pi_c^a(r_b) x^{cb} = 
\sum_{b \in G} r_b x^{ba} - \sum_{b \in G} r_b x^{ab} = 0.$
\end{proof}

\begin{prop}\label{bracket}
Suppose that $s \in S^G$. Then $[s,S] \subseteq [s,R]S$.
In particular, if $[s,R] = \{ 0 \}$, then $s \in C(S)$.
\end{prop}

\begin{proof}
Suppose that $s = \sum_{a \in G} r_a x^a \in S^G$.
Take $t \in R$ and $b \in G$.
From Proposition~\ref{xanmnr} and Proposition~\ref{commute}, it follows that
$[s,t x^b] = s (t x^b)- (t x^b) s = 
(s t)x^b - t (x^b s) = (s t) x^b - t (s x^b) =
(st) x^b - (ts) x^b = [s,t] x^b \in [s,R]S.$
The last part of the statement follows immediately.
\end{proof}

\begin{prop}\label{prop:M4properties}
The following assertions hold:
\begin{itemize}

\item[(a)] (D7) holds if and only if for every $a \in G$,
$(x^a , S^G , R) = \{ 0 \}$;

\item[(b)] (D8) holds if and only if for every $a \in G$,
$(x^a , R , S^G) = \{ 0 \}$.

\end{itemize}
\end{prop}

\begin{proof}
Take $a,b \in G$, $r \in R^G$ and $t \in R$.

(a) Suppose that (D7) holds. Then,
using Proposition~\ref{xanmnr} and Proposition~\ref{commute}
we get that
$(x^a , rx^b , t) = 
(x^a \cdot rx^b)t - x^a( rx^b \cdot t) =
( (x^a r) x^b ) t - x^a( r(x^b t) ) 
= (r x^a)(x^b t) - x^a (r (x^b t) ) = 
\sum_{c \in G} \left( (rx^a)(\pi_c^b(t) x^c) - 
x^a ( r \pi_c^b(t) x^c ) \right) 
= \sum_{c \in G} \sum_{d \in G} \left( r \pi_d^a( \pi_c^b(t) ) x^{dc} -
\pi_d^a( r \pi_c^b(t) ) x^{dc} \right) = 0.$
Now suppose that the equality $(x^a , S^G , R) = \{ 0 \}$
holds for any $a \in G$. Then, from Proposition~\ref{xanmnr}
and Proposition~\ref{commute}, we get that
$0 = (x^a , r , t) = 
(x^a r)t - x^a(rt) = 
(r x^a)t - x^a(rt) = 
r ( x^a t) - x^a(rt) = 
\sum_{c \in G}  \left( r \pi_c^a(t)x^c - \pi_c^a(rt) x^c \right)$
from which (D7) follows.

(b) Suppose that (D8) holds. 
We wish to show that $(x^a , t , rx^b) = 0$.
From Proposition~\ref{xanmnr}, it follows that it is
enough to show this equality for $b = 0$.
Now, using Proposition~\ref{xanmnr} and Proposition~\ref{commute}, we get that
$(x^a , t , r) = 
(x^a t)r - x^a (tr) = 
\sum_{c \in G} \left( (\pi_c^a(t) x^c ) r - \pi_c^a(tr) x^c \right) =
\sum_{c \in G} ( \pi_c^a(t)r - \pi_c^a(tr)) x^c = 0.$
On the other hand, if the equality $(x^a , R , S^G) = \{ 0 \}$ holds,
then, in particular,
$0 = (x^a , t , r) = 
(x^a t)r - x^a (tr) = 
\sum_{c \in G} \left( (\pi_c^a(t) x^c ) r - \pi_c^a(tr) x^c \right) =
\sum_{c \in G} ( \pi_c^a(t)r - \pi_c^a(tr)) x^c$,
which shows that (D8) holds.
\end{proof}

\begin{prop}\label{leftbracket}
Suppose that $s \in S$. Then $(s,S,S) \subseteq (s,R,R)S$.
In particular, if $(s,R,R) = \{ 0 \}$, then $s \in N_l(S)$.
\end{prop}

\begin{proof}
Suppose that $s = \sum_{a \in G} r_a x^a \in S$.
Take $r,t \in R$ and $b,c \in G$.
We wish to show that $(s , rx^b , tx^c) \subseteq (s,R,R)S$.
From Proposition~\ref{xanmnr}, it follows that it is 
enough to prove this inclusion for $c=0$. 
Using Proposition~\ref{xanmnr} again, it follows that
$ (s , rx^b, t) = 
(s \cdot rx^b)t - s( rx^b \cdot t ) = 
( (s r) x^b ) t - s( r ( x^b t ) ) 
= (sr)(x^b t) - \sum_{d \in G} s ( r \pi_d^b(t) x^d ) = 
\sum_{d \in G} (sr) \pi_d^b(t) x^d - \sum_{d \in G} ( s ( r \pi_d^b(t) ) ) x^d 
= \sum_{d \in G} ( (sr)\pi_d^b(t) - s( r \pi_d^b(t) ) ) x^d =
\sum_{d \in G} (s,r,\pi_d^b(t)) x^d \subseteq (s,R,R)S.$
The last part of the statement follows immediately.
\end{proof}

\begin{prop}\label{parenthesis}
Suppose that (D7) (or (D8)) holds and $s \in S^G$.
Then $(S,s,S) \subseteq (R,s,R)S$ (or $(S,S,s) \subseteq (R,R,s)S$).
In particular, if $(R,s,R) = \{ 0 \}
$ (or $(R,R,s) = \{ 0 \}$), 
then $s \in N_m(S)$ (or $s \in N_r(S)$).
\end{prop}

\begin{proof}
Take $r,t \in R$ and $b,c \in G$.
Suppose that (D7) holds. We wish to show that $(rx^b , s , tx^c) \subseteq (R,s,R)S$.
From Proposition~\ref{xanmnr}, it follows that it is enough
to show this inclusion for $c = 0$. 
Now, using Proposition~\ref{xanmnr}, 
Proposition~\ref{commute} and Proposition~\ref{prop:M4properties}(a), we get that
$(rx^b , s , t) =
( rx^b \cdot s ) t - rx^b (s \cdot t) =
(r (x^b s))t - r ( x^b (st) ) =
(r (s x^b ) ) t - r ( (x^b s ) t ) = ( (rs) x^b ) t - r ( (s x^b ) t ) 
= (rs)(x^b t) - r ( s ( x^b t ) ) =
\sum_{d \in G} \left( (rs) \pi_d^b(t) x^d - r (s \pi_d^b(t) ) x^d \right) 
= \sum_{d \in G} ( (rs) \pi_d^b(t) - r (s \pi_d^b(t) ) ) x^d = 
\sum_{d \in G} (r , s , \pi_d^b(t) ) x^d \in (R , s , R)S.$
The second part follows immediately.
Now suppose that (D8) holds. We wish to show that $(rx^b , tx^c , s) \subseteq (R,R,s)S$.
From Proposition~\ref{xanmnr} and Proposition~\ref{commute}, we get that
$ (rx^b , tx^c , s) = 
(rx^b \cdot tx^c)s - rx^b ( tx^c \cdot s) =
(rx^b \cdot t)(x^c s) - rx^b (t (x^c s) ) 
= (rx^b \cdot t)(s x^c) - rx^b (t (s x^c)) = 
( (rx^b \cdot t) s ) x^c - rx^b ( (ts) x^c )  
= ( ( rx^b \cdot t ) s ) x^c - ((rx^b)(ts))x^c = 
(rx^b , t , s) x^c$.
Thus, it is enough to prove the inclusion 
$(rx^b , tx^c , s) \subseteq (R,R,s)S$ for $c=0$.
Now, using Proposition~\ref{xanmnr}, Proposition~\ref{commute} 
and Proposition~\ref{prop:M4properties}(b), we get that
\begin{align*}
(r x^b , t , s) &= 
(r x^b \cdot t)s - rx^b (t \cdot s) = 
(r x^b \cdot t)s - r ( x^b (ts) ) =
(r x^b \cdot t)s - r ( (x^b t) s ) \\
&= \sum_{d \in G} \left( ( r \pi_d^b(t) x^d ) s - r( \pi_d^b(t) x^d \cdot s ) \right) =
\sum_{d \in G} \left( ( r \pi_d^b(t) ) (x^d s) - r( \pi_d^b(t) (x^d s) ) \right) \\
&= \sum_{d \in G} \left( ( r \pi_d^b(t) ) (s x^d) - r( \pi_d^b(t) (s x^d) ) \right) =  
\sum_{d \in G} \left( (( r \pi_d^b(t) ) s) x^d - ( r ( \pi_d^b(t)  s))  x^d \right) \\
&= \sum_{d \in G} ( (( r \pi_d^b(t) ) s) - ( r( \pi_d^b(t)  s) ) ) x^d =
\sum_{d \in G} (r , \pi_d^b(t) , s) x^d \in (R , R , s)S.
\end{align*}
The second part follows immediately.
\end{proof}

\begin{lem}\label{intersection}
The following
equalities hold:
\begin{displaymath}
	Z(S) = C(S) \cap N_l(S) \cap N_m(S)
				= C(S) \cap N_l(S) \cap N_r(S)
				= C(S) \cap N_m(S) \cap N_r(S)
\end{displaymath}
\end{lem}

\begin{proof}
This holds for any non-associative ring $S$ (see e.g. \cite[Proposition 2.1]{nystedtoinertrichter}).
\end{proof}

\begin{cor}\label{corcenter}
If $s \in S^G$ and $\pi$ is a strong $G$-derivation on $R$, then $s \in Z(S)$ if and only if
$s$ commutes and associates with all elements of $R$.
\end{cor}

\begin{proof}
The ''only if'' statement is immediate.
The ''if'' statement follows from 
Proposition \ref{leftbracket}, Proposition \ref{parenthesis}
and Lemma \ref{intersection}.
\end{proof}

\section{Simplicity}\label{Sec:Simplicity}

In this section, we show the main result
of this article (see Theorem~\ref{newmaintheorem}). It gives
us necessary and sufficient conditions for simplicity
of a large class of differential monoid rings.

\begin{defn}\label{defGsimple}
Let $I$ be an ideal of $R$. We say that $I$ is \emph{$G$-invariant}
if for all $a,b \in G$, the inclusion $\pi_b^a(I) \subseteq I$ holds.
We say that $R$ is \emph{$G$-simple} if $\{ 0 \}$ and $R$ are 
the only $G$-invariant ideals of $R$.
In Remark~\ref{rem:extendPI} we extended the maps $\pi_b^a : R \rightarrow R$,
for $a,b \in G$, to additive maps $\tilde{\pi}^a_b : S \to S$.
Using those extensions we may speak of \emph{$G$-simplicity of $S$}.
\end{defn}

\begin{prop}\label{RGsimple}
If $S$ is $G$-simple, then $R$ is $G$-simple.
\end{prop}

\begin{proof}
Suppose that $I$ is a non-zero $G$-invariant ideal of $R$.
Let $J$ denote the non-zero additive group $\sum_{g \in G} I x^g$. 
It is clear that $J$ is $G$-invariant.
We claim that $J$ is an ideal of $S$.
If we assume that the claim holds, then, by $G$-simplicity
of $S$, we get that $J = S$, and, hence, that $I = R$.
Now we show the claim.
Take $r \in R$, $i \in I$ and $a,b \in G$.
Since $I$ is a $G$-invariant ideal of $R$, we get that 
$(r x^a)(i x^b) = \sum_{c \in G} r \pi_c^a(i) x^{cb} \in 
\sum_{c \in G} R I x^{cb} \subseteq J$.
On the other hand, we also get that
$(i x^b)(r x^a) = \sum_{c \in G} i \pi_c^b(r) x^{ca} \in 
\sum_{c \in G} I R x^{ca} \subseteq J$.
\end{proof}

\begin{prop}\label{ZSGfield}
If $S$ is $G$-simple and for all $a,b \in G$, the map 
$\tilde{\pi}^a_b$ is right (left) 
$Z(S)^G$-linear, then $Z(S)^G$ is a field.
\end{prop}

\begin{proof}
We show the ''right'' case.
The ''left'' case is shown in an analogous fashion and is therefore left to the reader.
It is not difficult to see that $Z(S)^G$ is a commutative ring.
Take a non-zero $s \in Z(S)^G$.
Let $I = Ss$. Then $I$ is a non-zero ideal of $S$.
Take $a,b \in G$. Then 
$\tilde{\pi}_b^a(I) = \tilde{\pi}_b^a(Ss) = \tilde{\pi}_b^a(S)s \subseteq Ss = I$.
Therefore $I$ is $G$-invariant.
By $G$-simplicity of $S$, we get that there is some $t \in S$
such that $st = ts = 1$.
We know that $t \in Z(S)$.
It remains to show that $t \in S^G$.
Notice that $\tilde{\pi}_a^a(t) = \tilde{\pi}_a^a(t)st = \tilde{\pi}_a^a(ts)t =
\tilde{\pi}_a^a(1)t = 1t = t$,
and if $a \neq b$, then
$\tilde{\pi}_b^a(t) = \tilde{\pi}_b^a(t)st = \tilde{\pi}_b^a(ts)t = \tilde{\pi}_b^a(1)t = 0t = 0$.
\end{proof}

\begin{prop}\label{rightlinear}
If $a,b \in G$ and $\pi_b^a$ is right $R^G$-linear, then
$\tilde{\pi}_b^a$ is right $S^G$-linear.
\end{prop}

\begin{proof}
If $a,b,d,g \in G$, $r \in R$ and $s \in R^G$,
then $\tilde{\pi}_b^a( r x^d \cdot s x^g ) = 
\tilde{\pi}_b^a ( \sum_{c \in G} r \pi_c^d(s) x^{cg} ) =
\tilde{\pi}_b^a ( r s x^{dg} ) = 
\pi_b^a (rs) x^{dg} =
\pi_b^a (r)s x^{dg} =
\sum_{c \in G} \pi_b^a(r) \pi_c^d (s) x^{cg} =
\pi_b^a(r) x^d \cdot s x^g =
\tilde{\pi}_b^a (r x^d) s x^g.$
\end{proof}

\begin{defn}
We say that $\pi$ is \emph{commutative} if
$\pi_b^a \circ \pi_d^c = \pi_d^c \circ \pi_b^a$,
for all
$a,b,c,d \in G$.
\end{defn}

\begin{prop}\label{leftlinear}
Take $a,b \in G$.
If $\pi_b^a : R \rightarrow R$ is left $R^G$-linear
and $\pi$ is commutative, then
$\tilde{\pi}_b^a$ is left $S^G$-linear.
\end{prop}

\begin{proof}
If $d,e \in G$, $r \in R$ and $s \in R^G$,
then
$\tilde{\pi}_b^a(s x^e \cdot r x^d ) = 
\tilde{\pi}_b^a ( \sum_{c \in G} s \pi_c^e(r) x^{cd} ) =
\sum_{c \in G} \pi_b^a(s \pi_c^e(r)) x^{cd} =
\sum_{c \in G} s \pi_b^a(\pi_c^e(r)) x^{cd} =
\sum_{c \in G} s \pi_c^e(\pi_b^a(r)) x^{cd} =
s x^e \cdot \pi_b^a (r) x^d
=s x^e \cdot \tilde{\pi}_b^a (r x^d).$
\end{proof}

\begin{defn}\label{defrelation}
We say that $\pi$ is {\it well-ordered} if there is a well-order
$\preceq$ on $G$ with the property that if $a,b \in G$
satisfy $a \prec b$, then $\pi^a_b = 0$
(here, $a \prec b$ means that $a \preceq b$ and $a \neq b$).
In that case, we can define the {\it degree map} 
${\rm deg} : S \setminus \{ 0 \} \rightarrow G$ in the following way.
Take a non-zero element $s = \sum_{g \in G} r_g x^g \in S$.
Since ${\rm supp}(s) := \{ g \in G \mid r_g \neq 0 \}$ is finite,
${\rm supp}(s)$ has a greatest element with respect to $\preceq$.
That element will be denoted by ${\rm deg}(s)$.
\end{defn}

\begin{thm}\label{newmaintheorem}
Suppose that $S=R[G;\pi]$ is a differential monoid ring with $G$ commutative, and $\pi$ strong and well-ordered.
If (D8) holds (or (D7) holds and $\pi$ is commutative), then $S$ is $G$-simple if and only if
$R$ is $G$-simple and $Z(S)^G$ is a field.
\end{thm}

\begin{proof}
The ''only if'' direction
follows from 
Proposition~\ref{RGsimple}, 
Proposition~\ref{ZSGfield},
Proposition~\ref{rightlinear} and 
Proposition~\ref{leftlinear}.
Now we show the ''if'' direction (using either assumption). 
Suppose that $R$ is $G$-simple and that $Z(S)^G$ is a field.
Take a non-zero $G$-invariant ideal $I$ of $S$.
Since $\preceq$ is a well-order we can define
the least degree $m$ of non-zero elements of $I$.
Define the non-empty subset $J$ of $R$ by saying that
$r \in J$ if there are $r_g \in R$, for $g \in G$,
with $g \prec m$, such that $r x^m + \sum_{g \prec m} r_g x^g \in I$.
It is clear that $J$ is a non-zero left ideal of $R$.
Using that $\pi$ is well-ordered and $S$ is a differential monoid ring,
it follows that $J$ is also a right ideal of $R$.
Since $I$ is $G$-invariant it follows that $J$ is $G$-invariant.
From $G$-simplicity of $R$ we get that $J = R$.
In particular, we get that there are $r_g \in R$, for $g \in G$
with $g \prec m$, such that
$y := x^m + \sum_{g \prec m} r_g x^g \in I$.
Since $I$ is $G$-invariant, the minimality of $m$
yields that $y \in S^G$.
Take $r \in R$. Put $z = ry - yr$. Then ${\rm deg}(z) \prec m$
which implies that $z = 0$, since $z \in I$.
Thus, $y$ commutes with all elements of $R$.
Next we show that $y \in N(S)$.
By Proposition~\ref{corcenter}, 
it is enough to show that 
$y$ associates with all elements of $R$.
Take $r,s \in R$. It is easy to see that 
the degrees of all the elements
$(y,r,s)$, $(r,y,s)$ and $(r,s,y)$ are less than $m$.
Hence, by minimality of $m$, we get that they are all zero.
Therefore, $y$ is a non-zero element in the field $Z(S)^G$.
This implies that $I = S$.
Thus, $S$ is $G$-simple. 
\end{proof}

\begin{exmp}
In this example, $R$ denotes a unital non-associative algebra over a field $F$
such that $F \subseteq R$ and $1=1_R \in F$.
Furthermore, $S = R[G ; \pi]$ denotes the corresponding differential monoid ring
where $\pi$ is assumed to be a unital $G$-derivation.

Suppose that $G = \{ 0 , g \}$ is the unique cyclic monoid with
two elements not forming a group. Furthermore, suppose that $\pi_0^g \neq 0$. Then $g + g = g$.
Notice that there is a natural well-order on $G$.
From (D2), (D4) and (D6) it follows that 
$\pi_0^0 = \pi_g^g =  {\rm id}_R$,
$\pi_0^g(1) = 0$,
$ \pi_0^g \circ \pi_0^g = \pi_0^g$, and  ${\rm char}(F)=2$.
The last conclusion follows from the equality 
\begin{displaymath}
	\id_R = \pi^g_g = \pi^{g+g}_g = \pi^g_g \circ \pi^g_0 + \pi^g_0 \circ \pi^g_g + \pi^g_g \circ \pi^g_g
\end{displaymath}
which simplifies to
$\id_R = 2 \pi_0^g +\id_R$.

Put $P = \pi_0^g$.
We can write $R = {\rm ker}(P) \oplus V$ for some $F$-vector subspace $V$ of $R$
and notice that $R^G = {\rm ker} ( P ) = F$.
Thus, it follows that $\pi_0^g$ is both right and left $R^G$-linear.
Hence $\pi$ is strong.
It is easy to check that $Z(S)^G = F$.
By Theorem~\ref{newmaintheorem},
$S$ is $G$-simple if and only if 
$R$ is $G$-simple.
It now becomes easy to construct
examples of simple non-associative differential monoid rings.
Indeed, let $Q : R \rightarrow F$
denote the projection.
If for any non-zero $r \in R$, there is $r' \in R$
such that $Q(rr') \neq 0$ or $Q(r'r) \neq 0$,
then $S$ is simple.
\end{exmp}

\section{The monoid $\mathbb{N}^{(I)}$}\label{Sec:Malm+Voskoglou}

Let $\mathbb{N}$ denote the set of non-negative integers.
If $m,n \in \mathbb{N}$ satisfy $m \geq n$, then we let 
${m \choose n}$ have its usual meaning. 
If $m,n \in \mathbb{N}$ satisfy $m < n$, then we put ${m \choose n} = 0$.
Throughout this section, let $I$ be a set which is well-ordered with respect to a relation $\ll$.
For a function $f : I \rightarrow \mathbb{N}$
we let ${\rm supp}(f)$ denote the support of $f$, i.e.
the set $\{ i \in I \mid f(i) \neq 0 \}$. 
Let $\mathbb{N}^{(I)}$ denote the set 
of functions $I \rightarrow \mathbb{N}$ with finite support.
Take $f,g \in \mathbb{N}^{(I)}$. Define $f + g \in \mathbb{N}^{(I)}$
from the relations $(f + g)(i) = f(i) + g(i)$, for $i \in I$.
With this operation $\mathbb{N}^{(I)}$ is a commutative monoid.
Furthermore, put ${f \choose g} = \prod_{i \in I} {f(i) \choose g(i)}$.
We define a partial order $\leq$ on $\mathbb{N}^{(I)}$ by saying 
that $f \leq g$ if, for all $i \in I$, the relation $f(i) \leq g(i)$ holds.
Notice that if $f \leq g$, then the pointwise subtraction $g-f$
belongs to $\mathbb{N}^{(I)}$.
Given $f \in \mathbb{N}^{(I)}$, with cardinality of ${\rm supp}(f)$
equal to $m \in \mathbb{N}$, we will often,
for simplicity of notation, assume that 
${\rm supp}(f) = \{ 1,\ldots,m \} \subseteq I$.

\begin{defn}
Suppose that for each $i \in I$,
$\delta_i : R \rightarrow R$ is an additive map
satisfying $\delta_i(1)=0$.
Put $\Delta = \{ \delta_i \}_{i \in I}$
and $R_{\Delta} = \cap_{i \in I} {\rm ker}(\delta_i)$.
An ideal $J$ of $R$ is said to be \emph{$\Delta$-invariant}
if for each $i \in I$, $\delta_i(J) \subseteq J$.
If $\{ 0 \}$ and $R$ are the only $\Delta$-invariant ideals of $R$,
then $R$ is said to be \emph{$\Delta$-simple}.
We say that $\Delta$ is \emph{commutative}
if for all $i,j \in I$, the relation 
$\delta_i \circ \delta_j = \delta_j \circ \delta_i$ holds.
Furthermore, $\Delta$ is said to be \emph{a set of left (right) kernel derivations 
on $R$} if for each $i \in I$, $\delta_i$
is left (right) $R_{\Delta}$-linear.
Take $f,g \in \mathbb{N}^{(I)}$ and $m \in \mathbb{N}$
such that the support of $f$ is contained in $\{ 1 , \ldots , m \}$.
Define $\delta^f : R \rightarrow R$ by 
$\delta^f(r) = ( \delta_1^{f(1)} \circ \cdots \circ \delta_m^{f(m)} )(r)$, for $r \in R$.
Furthermore
we define $\pi_g^f : R \rightarrow R$ by
$\pi_g^f (r) = {f \choose g} \delta^{f-g}(r)$, for $r \in R$.
Notice that if $g \nleq f$, then $f-g$ is not defined as an element of
$\mathbb{N}^{(I)}$. However, since, in that case, ${f \choose g}=0$,
the term $\pi_g^f(r)$ is supposed to be interpreted as $0$.
\end{defn}

\begin{prop}\label{prop:NIproperties}
With the above notation the following assertions hold:
\begin{itemize}

\item[(a)] $\pi$ satisfies (D0), (D1), (D2) and (D6);

\item[(b)] $\pi$ satisfies (D4) $\Leftrightarrow$ 
$\pi$ is commutative $\Leftrightarrow$ $\Delta$ is commutative;

\item[(c)] Suppose that $\Delta$ is commutative. Then, 
$\pi$ satisfies (D5) $\Leftrightarrow$
each $\delta_i$, $i \in I$, is a derivation on $R$;

\item[(d)] $R^{ \mathbb{N}^{(I)} } = R_{\Delta}$;

\item[(e)] $\pi$ satisfies (D7) (or (D8)) $\Leftrightarrow$
$\Delta$ is a set of left (or right) kernel derivations on $R$;

\item[(f)] Suppose that $\Delta$ is commutative. Then $R$
is $\Delta$-simple $\Leftrightarrow$ $R$ is $\mathbb{N}^{(I)}$-simple.
\end{itemize}
\end{prop}

\begin{proof}
(a) and (b) follow immediately from the definition of $\pi$.

(c) Suppose that (D5) holds. Take $i \in I$ and define 
$f_i \in \mathbb{N}^{(I)}$ by the relations $f_i(j) = \delta_{i,j}$, for $j \in I$. 
For any $r,s\in R$ we get
$\delta_i (rs) = 
{f_i \choose 0} \delta^{f_i} (rs) = 
\pi_0^{f_i}(rs) = 
\pi_0^{f_i}(r) \pi_0^0(s) + \pi_{0}^{0}(r) \pi_0^{f_i}(s) =
\delta_i(r)s + r \delta_i(s)$,
which implies that $\delta_i$ is a derivation on $R$.

Now suppose that for each $i \in I$, $\delta_i$ is a derivation on $R$.
We first prove that, for any $i \in I$, $r,s \in R$ and $n\in \mathbb{N}$, the equation 
\begin{equation} \label{deltapowern} 
\delta_i^{n}(rs) = \sum_{k=0}^n \binom{n}{k} \delta^{n-k}(r) \delta^{k}(s)
\end{equation}
holds.
Take $i\in I$ and $r,s\in R$. Equation~\eqref{deltapowern} 
clearly holds if $n=0$ or $n=1$. We will prove the general case by induction. To this end, suppose that
Equation~\eqref{deltapowern}
holds for $n$. For clarity, we will write $\delta$ instead of $\delta_i$.
We now get 
\begin{align*}
\delta^{n+1}(rs) &= \delta(\delta^n(rs)) = \delta\left(\sum_{k=0}^n \binom n k \delta^{n-k}(r)\delta^k(s)\right) \\
&= \left(\sum_{k=0}^n \binom n k \left( \delta^{n+1-k}(r)\delta^k(s)+\delta^{n-k}(r)\delta^{k+1}(s)\right)\right) \\
&= \sum_{k=0}^n \binom n k \delta^{n+1-k}(r)\delta^k(s) +\sum_{k=0}^n \binom n k \delta^{n-k}(r)\delta^{k+1}(s) \\
&= \sum_{k=0}^n \binom n k \delta^{n+1-k}(r)\delta^k(s) +\sum_{k=1}^{n+1} \binom {n}{k-1} \delta^{n+1-k}(r)\delta^{k}(s) \\
&= \sum_{k=0}^{n+1} \left( \binom{n}{k} +\binom{n}{k-1}\right) \delta^{n+1-k}(r)\delta^k(s) =
\sum_{k=0}^{n+1} \binom{n+1}{k} \delta^{n+1-k}(r)\delta^k(s) 
\end{align*}
and by induction we conclude that Equation~\eqref{deltapowern} holds for any $n\in \N$.

We will first show (D5) in a special case. 
Suppose that $f(i) =n$, $g(i)=m \leq n$ and $f(j)=g(j) =0$ if $i \neq j$. Then we get
\begin{align*}
\pi_g^f(rs) &= \binom{n}{m} \delta^{n-m}(rs) = \binom n m \sum_{k=0}^{n-m} \binom{n-m}{k} \delta^{n-m-k}(r)\delta^k(s) \\
&=  \sum_{k=0}^{n-m} \binom n m \binom{n-m}{k} \delta^{n-m-k}(r)\delta^k(s) = \sum_{k=m}^n \binom n m \binom{n-m}{k-m} \delta^{n-k}(r) \delta^{k-m}(s) \\
&= \sum_{k=m}^n \binom n m \binom{n-m}{n-k} \delta^{n-k}(r) \delta^{k-m}(s) = \sum_{k=m}^n \binom{n}{n-k}\binom{k}{m} \delta^{n-k}(r) \delta^{k-m}(s) \\
&= \sum_{k=m}^n \binom{n}{k}\binom{k}{m} \delta^{n-k}(r) \delta^{k-m}(s)=\sum_{k=m}^n \binom{n}{k}\delta^{n-k}(r) \binom{k}{m} \delta^{k-m}(s)   =\sum_{ h \in \N^{(I)}} \pi_ h^f(r) \pi_g^h(s). 
\end{align*}
In the above calculation we have used an identity for binomial coefficients that will be generalized in Proposition~\ref{combinatorial}.

We have proved that (D5) holds for $\pi_g^f$ if the support of $f$ only contains one element. 
The general case can be proved by induction on the size of the support. 
To this end, suppose that we have proved (D5) if the functions involved 
have support of size at most $n$. 
Let $f$ be a function with a support of size $n+1$ and let $g\leq f$. 
We can write 
$f= f' +\bar{f}$ and $g= g'+\bar{g}$, where the support of $f'$ and $\bar{f}$ are disjoint, 
the support of $\bar{f}$ has size $n$, 
the support of $f'$ has size $1$ and $g' \leq f'$ and $\bar{g} \leq \bar{f}$. 
Then $\pi_g^f = \pi_{g'}^{f'} \circ \pi_{\bar{g}}^{\bar{f}}.$ By the induction hypothesis we have
$\pi_g^f(rs) = \pi_{g'}^{f'} \circ \pi_{\bar{g}}^{\bar{f}} (rs) = \pi_{g'}^{f'} \left( \sum_{\bar{h} } \pi_{\bar{h}}^{\bar{f}}(r) \pi_{\bar{g}}^{\bar{h}}(s) \right) = \sum_{  \bar{h} } \sum_{h'} \pi_{h'}^{f'}(\pi_{\bar{h}}^{\bar{f}}(r))\pi_{g'}^{h'}( \pi_{\bar{g}}^{\bar{h}}(s)) 
= \sum_{g \leq h'+\bar{h} \leq f} \pi_{h'}^{f'}(\pi_{\bar{h}}^{\bar{f}}(r))\pi_{g'}^{h'}( \pi_{\bar{g}}^{\bar{h}}(s))  \sum_{g \leq h\leq f} \pi_h^f(r)\pi_g^h(s) =\sum_{ h} \pi_h^f(r)\pi_g^h(s).$

(d) First we show the inclusion 
$R^{ \mathbb{N}^{(I)} } \supseteq R_{\Delta}$.
Take $r \in R_{\Delta}$. Then, for each $i \in I$,
the equality $\delta_i(r)=0$ holds.
Take $f,g \in \mathbb{N}^{(I)}$ such that $f \geq g$.
Then, from the definition of $\pi_g^f$, it follows that $\pi_g^f(r) = 0$.
Thus, $r \in R^{ \mathbb{N}^{(I)} }$.
Now we show the converse inclusion $R^{ \mathbb{N}^{(I)} } \subseteq R_{\Delta}$.
Take $r \in R^{ \mathbb{N}^{(I)} }$ and $i \in I$.
Define $f_i \in \mathbb{N}^{(I)}$ by the relations
$f_i(j) = 1$, if $j = i$, and $f_i(j) = 0$, otherwise.
Since $r \in  R^{ \mathbb{N}^{(I)} }$, we get, in particular,
that $0 = \pi_0^{f_i}(r) = {f_i\choose 0} \delta^{f_i}(r) = \delta_i(r)$.
Thus, $r \in {\rm ker}(\delta_i)$. Hence $r \in R_{\Delta}$.

(e) This follows immediately from (d).

(f) If $R$ is $\mathbb{N}^{(I)}$-simple, then clearly $R$ is also $\Delta$-simple.
Now suppose that $R$ is not $\mathbb{N}^{(I)}$-simple. We want to show that $R$ is not $\Delta$-simple.
Let $J$ be a non-zero proper $\mathbb{N}^{(I)}$-invariant ideal of $R$.
Take $i\in I$.
Define $f_i \in \mathbb{N}^{(I)}$ as in the proof of (c).
For any $r\in J$ we get
$\delta_i(r) = {f_i\choose 0} \delta^{f_i}(r) \in J$.
This shows that $J$ is $\Delta$-invariant and hence $R$ is not $\Delta$-simple.
\end{proof}

\begin{defn}
Now we will define a well-order $\preceq$ on $\mathbb{N}^{(I)}$
which extends $\leq$.
To this end, take $f,g \in \mathbb{N}^{(I)}$ with $f \neq g$ and
put $|f| = \sum_{i \in I} f(i)$.
Case 1: if $|f| > |g|$, then put $f \succ g$.
Case 2: if $|f| < |g|$, then put $f \prec g$.
Case 3: Suppose that $|f|=|g|$.
Then there is $j \in I$ such that 
$f(i)=g(i)$, for  $i \gg j$,
but $f(j) \neq g(j)$.
If $f(j) > g(j)$, then put $f \succ g$.
If $f(j) < g(j)$, then put $f \prec g$.
We will refer to $\preceq$ as the 
graded lexicographical ordering on $\mathbb{N}^{(I)}$.
\end{defn}

\begin{rem}
For each $i\in I$, the map $\delta_i : R \to R$
may be extended to an additive map $\tilde{\delta}_i : S \to S$
by defining $\tilde{\delta}_i ( \sum_{f\in \mathbb{N}^{(I)}} r_f x^f) = \sum_{f\in \mathbb{N}^{(I)}} \delta_i(r_f) x^f$.
\end{rem}

\begin{thm}\label{cornewtheorem}
If $\Delta$ is a commutative set of left (right) 
kernel derivations on $R$, then the 
differential monoid ring $S = R[ \mathbb{N}^{(I)} ; \pi]$
is simple if and only if 
$R$ is $\Delta$-simple and $Z(S)$ is a field.
\end{thm}

\begin{proof}
Put $G = \mathbb{N}^{(I)}$ and equip $G$ with 
the graded lexicographical ordering.
All ideals $J$ of $S$ are $G$-invariant.
Indeed, take $s = \sum_{f \in G} r_f x^f \in J$ and $i \in I$.
Then $J \ni x^{f_i} s - s x^{f_i} = 
\sum_{f \in G} \delta_i(r_f) x^f = \tilde{\delta}_i(s)$.
By induction, $\pi_g^f(s) \in J$ for all $f,g \in G$
with $f \geq g$.
This implies that $G$-simplicity of $S$ is 
equivalent to simplicity of $S$.
Also, $Z(S)^G = Z(S)$. In fact,
given $s = \sum_{f \in G} r_f x^f \in Z(S)$ and $i \in I$,
we have $0 = x^{f_i} s - s x^{f_i} = 
\sum_{f \in G} \delta_i(r_f) x^f = \tilde{\delta}_i(s)$.
The claim now follows immediately from Theorem~\ref{newmaintheorem}
and Proposition~\ref{prop:NIproperties}(f).
\end{proof}

\begin{rem}
Theorem \ref{cornewtheorem}
generalizes \cite[Theorem 4.15]{oinert2013}
both to the case of several variables and to 
the non-associative situation.
In fact, 
if $R$ is associative, $I$ is the finite set 
$\{ 1,\ldots,n \}$ and each $\delta_i$, for $i \in I$,
is a derivation on $R$, then $R[\mathbb{N}^{(I)} ; \pi]$ coincides with the
differential polynomial ring $R[x_1,\ldots,x_n ; \delta_1,\ldots,\delta_n]$
(see Section~\ref{sec:one}).
\end{rem}

Now we wish to proceed to prove non-associative generalizations
(see Theorem \ref{maintheorem0} and Theorem \ref{maintheoremp})
of results by Voskoglou \cite{voskoglou1985} and Malm \cite{malm1988}
for
differential polynomial rings 
in, possibly, infinitely many variables.
Therefore, for the rest of this section, we assume
that $\Delta$ is a commutative set of left (or right) kernel derivations on
a (possibly non-associative) ring $R$
and we put $S = R[\mathbb{N}^{(I)} ; \pi ]$. 
To this end, we first prove a few useful propositions.

\begin{prop}\label{vandermonde}
If $f,g,h \in \mathbb{N}^{(I)}$ are chosen 
so that $f = g + h$, then, for every $l \in \mathbb{N}^{(I)}$,
we get that
$ \sum_{p+q=l} {g \choose p} {h \choose q}  = {f \choose l}$.
\end{prop}

\begin{proof}
This follows from Vandermonde's identity:
$ \sum_{q+r=s} {b \choose q} {c \choose r}  = {a \choose s}$
which holds for all $a,b,c,q,r,s \in \mathbb{N}$.
\end{proof}

\begin{defn}
If $f \in \mathbb{N}^{(I)}$, then
we put 
$(-1)^f = (-1)^{|f|}$. 
\end{defn}

\begin{prop}\label{right}
If $r \in R$ and $f \in \mathbb{N}^{(I)}$, then
$r x^f = \sum_{g\leq f} (-1)^g {f \choose g} x^{f-g} \delta^g(r)$. 
\end{prop}

\begin{proof}
Suppose that the cardinality of the support of $f$ is $n$. 
We will show the claim by induction over $n$.

Base case: $n=1$. This has already been proven in \cite[Proposition 3.7]{nystedtoinertrichter}.

Induction step: Suppose that $n > 1$ and that the claim holds for all
elements in $\mathbb{N}^{(I)}$ such that the cardinality of the support of 
the element is less than $n$.
Suppose that $f = g + h$, for some $g,h \in \mathbb{N}^{(I)}$,
is chosen so that the cardinalities of the supports of $g$ and $h$ are less than $n$.
From the induction hypothesis and Proposition~\ref{vandermonde}, 
we now get that
\begin{align*}
r x^f &= 
r x^{g+h} = 
(r x^g)x^h = 
\sum_{p \leq g} (-1)^p {g \choose p} x^{g-p} \delta^p(r) x^h \\
&= \sum_{p \leq g, \ q \leq h} (-1)^p (-1)^q {g \choose p} {h \choose q} 
x^{g-p} x^{h-q} (\delta^p \circ \delta^q)(r) \\
&=
\sum_{p \leq g, \ q \leq h} (-1)^{p+q} {g \choose p} {h \choose q} x^{f - (p+q)} \delta^{p+q}(r) \\
&= \sum_{l \leq f} \sum_{p+q=l} (-1)^l {g \choose p} {h \choose q} x^{f - l} \delta^l(r) =
\sum_{l \leq f} (-1)^l {f \choose l} x^{f - l} \delta^l(r).
\end{align*}
\end{proof}

\begin{prop}\label{combinatorial}
If $f,g,h \in \mathbb{N}^{(I)}$, then
${f \choose g} {f-g \choose h} = {f \choose h} {f-h \choose g}$.
\end{prop}

\begin{proof}
This follows from the well known equality
${a \choose b} {a-b \choose c} = {a \choose c} {a-c \choose b}$
which holds for all $a,b,c \in \mathbb{N}$.
\end{proof}

\begin{prop}\label{less}
If $a \in Z(S)$, then, for each $g \in \mathbb{N}^{(I)}$, 
we get that $\sum_{f \geq g} x^{f-g} {f \choose g} a_f \in Z(S)$.
\end{prop}

\begin{proof}
Put $b = \sum_{f \geq g} x^{f-g} {f \choose g} a_f$.
Using Proposition~\ref{prop:NIproperties}(e), it is enough to check the conditions in Corollary~\ref{corcenter}.
Since $a$ commutes with every $x^h$, for $h \in \mathbb{N}^{(I)}$,
we get that $a_f \in R_{\Delta}$, for $f \in \mathbb{N}^{(I)}$.
Next we show that $b$ commutes with 
every $r \in R$. 
Using that $a_f \in R_{\Delta}$, for $f \in \mathbb{N}^{(I)}$,
we can write $a = \sum_{f \in \mathbb{N}^{(I)} } x^f a_f$.
Since $a r = r a$, we may use Proposition~\ref{right} to conclude that
\begin{align*}
\sum_f x^f a_f r &= ar = ra =
\sum_f r x^f a_f = 
\sum_{f,g} (-1)^g {f \choose g} x^{f-g} \delta^g(r) a_f \\
&= \sum_{f,h} (-1)^{f-h} {f \choose f-h} x^h \delta^{f-h}(r) a_f,
\end{align*}
where we in the last sum have put $f-g=h$.
Hence, for each $h \in \mathbb{N}^{(I)}$, we get that
\begin{equation}\label{identity}
a_h r = \sum_f (-1)^{f-h} {f \choose f-h} \delta^{f-h}(r) a_f.
\end{equation}
Thus
$	r b = \sum_f r x^{f-g} {f \choose g} a_f = 
\sum_{f,h} x^{f-g-h} (-1)^h {f-g \choose h} \delta^h(r) {f \choose g} a_f.$
If we now put $v = f-h$ in the last sum and use 
Proposition~\ref{combinatorial}, then we get 
\begin{displaymath}
\sum_{f,v} x^{v-g} {f \choose g} {f-g \choose f-v}
(-1)^{f-v} \delta^{f-v}(r) a_f = 
\sum_{f,v} x^{v-g} {f \choose f-v} {v \choose g} 
(-1)^{f-v} \delta^{f-v}(r) a_f.	
\end{displaymath}
Finally, using \eqref{identity}, the last sum equals
$\sum_v x^{v-g} {v \choose g} a_v r = b r.$

To conclude the proof, we show that $b$ associates with all elements in $R$.
From the relations $(R,R,a) = \{ 0 \}$ and $(a,R,R) = \{ 0 \}$
it follows that for each $f \in \mathbb{N}^{(I)}$, 
$(R,R,a_f) = (a_f,R,R) =  \{ 0 \}$.
Hence we get that $(R,R,b_f) = (b_f,R,R) = \{ 0 \}$,
for $f \in \mathbb{N}^{(I)}$.
Thus, $(b,R,R) = (R,R,b) = \{ 0 \}$.
Since $[b,R] = \{ 0 \}$, we automatically get that $(R,b,R) = \{ 0 \}$
(see the proof of \cite[Proposition 2.1]{nystedtoinertrichter}).
\end{proof}

\begin{defn}
Take $a = \sum_{f \in \mathbb{N}^{(I)} } a_f x^f \in S$.
Recall from Definition~\ref{defrelation} that 
${\rm deg}(a)$ is the largest $f \in \mathbb{N}^{(I)}$,
with respect to $\preceq$, such that $a_f \neq 0$.
If $a$ is non-zero and $a_{{\rm deg}(a)} = 1$,
then we say that $a$ is \emph{monic}.
We say that $a$ is \emph{constant} if ${\rm deg}(a) = 0$.
We say that $a$ is \emph{linear} if $a$ is non-constant,
$a_0 = 0$ and the set ${\rm supp} ( {\rm deg}(a) )$
contains exactly one element.
Notice that if $f \in \mathbb{N}^{(I)}$ has ${\rm supp}(f) = \{ i \}$
for some $i \in I$, then, for every $r \in R$,
the relation $x^f r = \delta_i(r) + r x^f$ holds.  
\end{defn}

\begin{prop}\label{mainprop0}
Suppose that $R$ is $\Delta$-simple and that ${\rm char}(R)=0$.
Put $F=Z(R)_{\Delta}$.
The following assertions hold:
\begin{itemize}

\item[(a)] If $Z(S)$ only contains constants, then $Z(S)=F$.

\item[(b)] If $Z(S)$ contains non-constants,
then $Z(S)$ contains a unique,
up to addition by elements from $F$,
non-constant monic $a$
of least graded lexicographical degree.
In that case, $Z(S)$ is not a field and 
there is $c \in R_{\Delta}$, $m \in \mathbb{N}$
and $c_i \in F$, for $i \in \{ 1,\ldots,m-1\}$,
such that $a = x_m + \sum_{i = 1}^{m-1} c_i x_i - c$ and, hence,
$\delta_c = \delta_m + \sum_{i=0}^{m-1} c_i \delta_i$.
\end{itemize}
\end{prop}

\begin{proof}
(a) This is clear.

(b) Suppose that $Z(S)$ is not contained in $R$.
Let $a$ be non-constant in $Z(S)$
of least graded lexicographical degree $h \in \mathbb{N}^{(I)}$.
For each $g \in \mathbb{N}^{(I)}$, let $b_g = \sum_{f \geq g} x^{f-g} {f \choose g} a_f$.
From Proposition~\ref{less} and the definition of $a$ 
it follows that if $g > 0$, then $a_g = b_g \in Z(S)$.
Thus, if $g > 0$, then $a_g \in F$.
In particular, $a_h \in F \setminus \{ 0 \}$.
We can thus, from now on, assume that $a_h = 1$
so that $a$ is monic.
We claim that $a - a_0$ is linear.
If we assume that the claim holds,
then, by the definition of the graded lexicographical ordering, 
there is
$c \in R_\Delta$,
$m \in \mathbb{N}$ and $c_i \in F$, 
for $i \in \{ 1,\ldots,m-1 \}$, such that 
$a = x_m + \sum_{i = 1}^{m-1} c_i x_i - c$.
From the relations $ar = ra$, for $r \in R$,
it follows that $\delta_c = \delta_m + \sum_{i=0}^{m-1} c_i \delta_i$.
Now we show the claim.
Seeking a contradiction, 
suppose that ${\rm supp} ( {\rm deg}(a) )$
contains more than one element.
Then there is $g \in \mathbb{N}^{(I)}$ such that $h > g > 0$.
Since ${\rm char}(R)=0$, we get that ${h \choose g} \in F \setminus \{ 0 \}$
and thus $0 < {\rm deg}(b_g) < {\rm deg}(a)$
which contradicts the choice of $a$.
\end{proof}

\begin{thm}\label{maintheorem0}
If ${\rm char}(R)=0$ and we put $F = Z(R)_{\Delta}$, then $S$ is simple
if and only if $R$ is $\Delta$-simple and no
non-trivial finite $F$-linear combination of 
elements from $\Delta$
is an inner derivation on $R$
defined by an element from $R_{\Delta}$.
\end{thm}

\begin{proof}
The ''if'' statement follows from Theorem~\ref{cornewtheorem}
and Proposition~\ref{mainprop0}.
Now we show the ''only if'' statement.
Suppose that $S$ is simple.
By Theorem~\ref{cornewtheorem}, $R$ is $\Delta$-simple
and $Z(S)$ is a field.
Suppose that there are $c_1,\ldots,c_n \in F$,
not all of them equal to zero, such that
$c_1 \delta_1 + \ldots + c_n \delta_n = \delta_c$.
Take $c \in R_{\Delta}$ and consider the element $p = c_1 x_1 + \ldots c_n x_n - c$.
Then it is clear that $p x_i = x_i p $ for all $i \in I$.
Now take $r \in R$. Then 
$p r = r p + rc +  c_1 \delta_1(r) + \ldots + c_n \delta_n(r) - cr = 
r p + c_1 \delta_1(r) + \ldots + c_n \delta_n(r) + rc - cr =
r p + c_1 \delta_1(r) + \ldots + c_n \delta_n(r) - \delta_c(r) = r p$.
Therefore $p \in Z(S)$. But, by Proposition~\ref{mainprop0}, this is a contradiction since $Z(S)$ is
a field and $p$ is non-constant.
\end{proof}

In the proof of the next proposition,
we will use the following notation.
Let $N = \{ -\infty\} \cup \mathbb{N}$
and let $p$ be a prime number.
We will formally write $p^{-\infty} = 0$.
Let $N^{(I)}$ denote the set of 
functions $f : I \rightarrow N$ with the property
that $f(i) = -\infty$ for all but finitely 
many $i \in I$.
Given $f \in N^{(I)}$, let $p^f \in \mathbb{N}^{(I)}$
be defined by $(p^f)(i) = p^{f(i)}$, for $i \in I$. 

\begin{prop}\label{mainpropp}
Suppose that $R$ is $\Delta$-simple and that ${\rm char}(R)=p$.
Put $F = Z(R)_{\Delta}$.
The following assertions hold:
\begin{itemize}

\item[(a)] If $Z(S)$ only contains constants, then $Z(S)=F$.

\item[(b)] If $Z(S)$ contains non-constants,
then $Z(S)$ contains a unique,
up to addition by elements from $F$,
non-constant monic $a$
of least graded lexicographical degree.
In that case,
if the maps $\delta_i^{p^j}$, for $i\in I$ and $j \in \N$,
are $F$-linearly independent,
then there is
$c \in R_\Delta$,
$m,n \in \mathbb{N}$
and $c_{ij} \in F$, for $(i,j) \in \{1,\ldots,m\} \times \{0,\ldots,n\}$,
such that $c_{mn}=1$ and
$a = \sum_{i=1}^m \sum_{j=0}^n c_{ij} x_i^{p^j} - c$.
Thus, $\delta_c = \sum_{i=1}^m \sum_{j=0}^n c_{ij} \delta_i^{p^j}$.
\end{itemize}
\end{prop}

\begin{proof}
(a) This is clear.

(b) Suppose that $Z(S)$ is not contained in $R$.
Let $a$ be a non-constant polynomial in $Z(S)$
of least graded lexicographical degree $h \in \mathbb{N}^{(I)}$.
For each $g \in \mathbb{N}^{(I)}$, let $b_g = \sum_{f\geq g} x^{f-g} {f \choose g} a_f$.
From Proposition~\ref{less} and the definition of $a$ 
it follows that if $g > 0$, then $a_g = b_g \in Z(S)$.
Thus, if $g > 0$, then $a_g \in F$.
In particular, $a_h \in F \setminus \{ 0 \}$.
We can thus, from now on, assume that $a_h = 1$
so that $a$ is monic.
Take $f \in \mathbb{N}^{(I)}$ such that $a_f \neq 0$.
From Proposition~\ref{less} and Lucas' theorem
it follows that $f = p^t$ for some $t \in N^{(I)}$.
We say that $t$ is {\it more than singly supported}
if there are $i,j \in I$, with $i < j$, such that 
$t(i) \neq -\infty \neq t(j)$.
Seeking a contradiction, suppose that the set
$Z = \{ t \in N^{(I)} \mid  \mbox{$t$ is more than singly supported and
$a_{p^{t}} \neq 0$ } \}$ is non-empty.
To this end, let $s$ denote the unique element from $Z$ 
of least graded lexicographical degree.
Thus, there are $i,j \in I$, with $i < j$,
such that $s(i) \neq -\infty \neq s(j)$ and
$s(k) = -\infty$, for $k > j$.
Given $n \in N$,
define $s_n \in N^{(I)}$ by the relations
$s_n(l) = s(l)$, for $l < j$, and $s_n(j) = n$, otherwise.
In the relation $ar-ra = 0$, considering the coefficient of 
$x^{p^{s_{-\infty}}}$ yields the relation
$	\sum_{n=0}^j
a_{ p^{s_n} }
\delta_n^{ p^{s_n(j)} } = 0$
which is a contradiction since $a_{ p^{s_j} } \ne 0$.
Hence $Z = \emptyset$ and the desired result follows.
\end{proof}

\begin{thm}\label{maintheoremp}
If ${\rm char}(R)=p$ and we put $F = Z(R)_{\Delta}$, then 
$S$ is simple if and only if $R$ is $\Delta$-simple and no
non-trivial finite $F$-linear combination of 
$\delta_i^{p^j}$, for $i \in I$ and $j \in \mathbb{N}$, 
is an inner derivation on $R$ defined by an
element from $R_{\Delta}$.
\end{thm}

\begin{proof}
The ''if'' statement follows from Theorem~\ref{cornewtheorem}
and Proposition~\ref{mainpropp}.
Now we show the ''only if'' statement.
Suppose that $S$ is simple.
By Theorem~\ref{cornewtheorem}, $R$ is $\Delta$-simple
and $Z(S)$ is a field.
Seeking a contradiction,
suppose that there are $c_{ij}  \in F$,
not all of them equal to zero, and
$c \in R_0$
such that
$\sum_{ij} c_{ij} \delta_i^{p^j} = \delta_c$.
Consider now the element $a = \sum_{ij} c_i^{p^j} x_i^{p^j} - c$.
Then it is clear that $a x_i = x_i a$ for all $i \in I$.
Take $r \in R$. Then 
$a r - r a = rc +  c_1 \delta_1(r) + \ldots + c_n \delta_n(r) - cr = 
c_1 \delta_1(r) + \ldots + c_n \delta_n(r) + rc - cr =
c_1 \delta_1(r) + \ldots + c_n \delta_n(r) - \delta_c(r) = 0$.
Therefore $a \in Z(S)$. But, by Proposition~\ref{mainpropp}, this is a contradiction since $Z(S)$ is
a field and $a$ is non-constant.
\end{proof}

\end{document}